\documentclass[reqno,11pt]{amsart}

\usepackage[top=2.0cm,bottom=2.0cm,left=3cm,right=3cm]{geometry}
\usepackage{amsthm,amsmath,amssymb,dsfont}
\usepackage{mathrsfs,amsfonts,functan,extarrows,mathtools,ulem}
\usepackage[colorlinks]{hyperref}
\usepackage{enumerate}
\usepackage{marginnote}
\usepackage{xcolor}
\usepackage{stmaryrd}
\usepackage{esint}
\usepackage{graphicx}
\usepackage{multirow}

\newtheorem{theorem}{Theorem}[section]

\newtheorem{thmpaux}{Theorem}  
\newenvironment{thmp}[1][]
{%
\begin{thmpaux}}
{\end{thmpaux}}

\newtheorem{thmppaux}{Theorem}  

\newtheorem{proposition}{Proposition}[section]

\newtheorem{remark}{Remark}[section]
\newtheorem{lemma}{Lemma}[section]

\numberwithin{equation}{section}
\allowdisplaybreaks

\def\p{\partial}

\def\vep{\varepsilon}

\def\ga{\gamma}
\def\Ga{\Gamma}
\def\Om{\Omega}
\def\om{\omega}
\def\al{\alpha}
\def\th{\theta}

\def\ka{\kappa}
\def\de{\delta}

\def\n{\mathbf{n}}

\def\R{\mathbb{R}}

\def\vphi{\varphi}
\def\vep{\varepsilon}

\def\div{\mathrm{div\,}}
\def\curl{\mathrm{curl}}

\def\v{\mathrm{v}}
\def\w{\mathrm{w}}

\def\u{\mathbf{u}}
\def\v{\mathbf{v}}
\def\w{\mathbf{w}}
\def\e{\mathbf{e}}

\def\L{\mathcal{L}}

\def\ml{\mathcal}
\def\mb{\mathbb}
\def\mf{\mathbf}

\newcommand{\Rmnum}[1]{\uppercase\expandafter{\romannumeral #1}}

\newcounter{wronumber}

\begin{document}
\title[steady incompressible Euler flow] {solutions to the two-dimensional steady
incompressible Euler equations in an annulus}

\author[W.G.Yang] {Wengang Yang}
\address[Wengang Yang]
{\newline School of Mathematics and Statistics, Wuhan University, Wuhan, 430072, P. R. China}
\email{yangwg@whu.edu.cn}

\begin{abstract}
This paper investigates the well-posedness of five classes of boundary value problems for the two-dimensional steady incompressible Euler equations in an annular domain. Three of these boundary conditions can be effectively addressed using the Grad-Shafranov method, and the well-posedness of solutions in the $C^{1,\al}$ space is established via variational techniques. We demonstrate that all five classes of boundary value problems are solvable through the vorticity transport method. Based on this approach, we further prove the well-posedness of $C^{2,\al}$ solutions under a perturbation framework.

\end{abstract}

\keywords{Steady Euler flows, Magneto-hydrostatic equations, Boundary value problem}
\subjclass[2020]{35M32, 35Q31, 76W05}
\date{}
\maketitle

\section{Introduction and main results}
The steady inviscid, incompressible fluid in a two-dimensional domain $\Om$ is governed by the Euler system,
\begin{equation}\label{InE_eq}
\begin{split}
\u\cdot\nabla \u+\nabla p&=0,\\
\div\u&=0,
\end{split}
\end{equation}
where $\u=(u_1,u_2):\Om\to\R^2$ stands for the velocity vector field and $p:\Om\to\R$ is the pressure. The domain $\Om$ concerned in this paper is 
\begin{equation*}
\Om:=\{(x_1,x_2):0<r_0<\sqrt{x_1^2+x_2^2}<r_1\}.
\end{equation*}
Since $\Om$ is an annulus, it is convenient to rewrite the equations \eqref{InE_eq} into polar coordinates. The velocity field $\u$ can be expressed in both Cartesian and polar coordinates as
\begin{equation*}
\u=u_1(x_1,x_2) \e_1+u_2(x_1,x_2) \e_2=u_r(r,\th){\e_r}+u_\th(r,\th) \e_\th,
\end{equation*}
where 
\begin{equation*}
\e_r=(\cos{\th},\sin{\th})^t,\,\e_\th=(\sin{\th},\cos{\th})^t.
\end{equation*} 

In polar coordinates, the domain $\Om$ can be rewritten as
\begin{equation*}
\Om=\{(r,\th):r_0<r<r_1,\,\th\in\mb{T}_{2\pi}\}.     
\end{equation*}

For simplicity, the pressure $p$ is still denoted in polar coordinates as $p=p(r,\th)$. Then the steady incompressible Euler equations \eqref{InE_eq} in polar coordinates read
\begin{equation}\label{InE_Pol}
\begin{split}
(u_r\p_r+\frac{u_\th}{r}\p_{\th})u_r-\frac{u_\th^2}{r}+\p_r p&=0,\\
(u_r\p_r+\frac{u_\th}{r}\p_{\th})u_\th+\frac{u_r u_\th}{r}+\frac{1}{r}\p_\th p&=0,\\
\frac{1}{r}\p_r(r u_r)+\frac{1}{r}\p_{\th}u_\th&=0.
\end{split}
\end{equation}

In the two-dimensional setting, the vorticity reduces to a scalar
\begin{equation*}
\om= \p_{x_1} u_2-\p_{x_2}u_1 =\frac{1}{r} \big( \p_r(r u_\th)-\p_\th u_r\big).
\end{equation*}

For the irrational flow, namely $\om=0$, there exists a potential function $\phi$ such that the velocity field $\u$ satisfies $\u=\nabla\phi$. In this case, the Euler system \eqref{InE_eq} reduces to the Laplace equation $\Delta\phi=0$. Consequently, the well-established theory of harmonic functions can be directly applied to analyze irrational solutions of the system \eqref{InE_eq}.	In contrast, this paper focuses on solutions with non-vanishing vorticity under appropriate boundary conditions. The study of such flows constitutes one of the most challenging and physically important topics in incompressible fluid dynamics \cite{Majda1986,MB2002}. The mathematical analysis of stationary solutions for incompressible Euler flows, including existence, stability, and topological properties, has been thoroughly studied in a large body of literature; see \cite{Bineau1972,EPS2018,EH2025,FB1974,FOG2007,GKM2025} and references therein.

This paper investigates the well-posedness of different boundary conditions for the steady Euler equations \eqref{InE_eq}. The study of such boundary value problems traces back to Grad and Rubin's seminal 1958 work \cite{GR1958}, which proposed physically meaningful boundary conditions for the magnetohydrostatics (MHS) equations. The MHS equations describe the equilibrium state of a plasma or conducting fluid in the presence of a magnetic field, under the fundamental assumption of vanishing fluid velocity. The MHS equilibrium is governed by the following system:
\begin{equation}\label{eq_MHS}
\mf{B}\times j=-\nabla p,\,\nabla\times \mf{B}=j,\,
\nabla\cdot \mf{B}=0.
\end{equation}
Here, $p$ is the plasma pressure, $\mf{B}$ the magnetic field, and $j$ the current density. The MHS model is fundamental to studying astrophysical plasmas and plasma confinement fusion \cite{P2014,Kul2005}. Indeed, the MHS equation is mathematically equivalent to the Euler system \eqref{InE_eq}. By using the vector identity
\begin{equation*}
(\u\cdot\nabla)\u=\frac{1}{2}\nabla|\u|^2-\u\times\om,
\end{equation*}
the Euler equation \eqref{InE_eq} can be reformulated as 
\begin{equation}\label{eq_u_om_H}
\u\times \om=\nabla H,\,
\nabla\times \u=\om,\,
\nabla\cdot \u=0.
\end{equation}
where $H=\frac{1}{2}|\u|^2+p$. Thus, the system \eqref{eq_u_om_H} constitutes a hydrodynamic analogue of the MHS equations \eqref{eq_MHS}, sharing similar mathematical structure while describing different physical regimes. For our analysis, we summarize in Table \ref{tab_BC} the boundary conditions for the Euler equations \eqref{InE_Pol} that will be investigated in this work.




\subsection{Main results and overview}
We establish the existence and uniqueness of solutions with non-zero vorticity for the steady incompressible Euler system \eqref{InE_Pol}, satisfying the boundary conditions presented in Table \ref{tab_BC}. Two principal methods are employed to solve these boundary value problems for the steady Euler: the Grad-Shafranov method \cite{G1967,Saf1966} and the vorticity transport method \cite{Alber1992}. Here we outline the key ideas underlying these arguments and give some remarks. 

The key ingredient of the Grad-Shafranov method is introducing a stream function $\psi(x)$ by using the incompressible condition in \eqref{InE_eq} such that $\u=\nabla_x^\perp\psi$. Then the first equation in  \eqref{eq_u_om_H} yields $H=F(\psi)$ since $\nabla^\perp \psi\cdot\nabla_x H=0$. The second equation  in  \eqref{eq_u_om_H} becomes
\begin{equation}\label{laplace_eq}
\Delta\psi=F'(\psi).
\end{equation}

Once the boundary conditions at the inner circle are specified according to the first line in Table \ref{tab_BC}, the source term $F(\psi)$ in the above equation becomes fully determined by these boundary data. When combined with the corresponding boundary condition at the outer circle, the analysis of the steady Euler equation \eqref{InE_eq} reduces to a boundary value problem for a second-order elliptic equation in terms of $\psi$. Consequently, standard techniques for elliptic equations become applicable in this setting. For our purposes, we will develop a more explicit reformulation of this procedure in polar coordinates in Section \ref{sec2}.
Building upon the fruitful techniques of elliptic equations, Grad-Shafranov reduction is extensively used to study the flexibility, rigidity \cite{CDG2021,HN2023}, and vertex structures \cite{Abe2022,Tur1989} in steady fluid motion. This approach was subsequently extended in \cite{BE2019}, where the authors developed a novel construction of rotational solutions through the introduction of two stream functions and implementation of Nash-Moser iteration.

\begin{table}[ht]
\centering
\begin{tabular}{|c|c|c|}
\hline
{Inner circle $(r=r_0)$} & Outer circle $(r=r_1)$ & Notation\\ \hline
\multirow{3}{*}{$\u\cdot \n=f_0(\th),\,\frac{1}{2}|\u|^2+p=b_0(\th)$} &
$\u\cdot \n=f_1(\th)$&\eqref{BC1} \\ \cline{2-3} 
& $\frac{1}{2}|\u|^2+p=b_1(\th)$&\eqref{BC2}\\ \cline{2-3}
& $p=p_1(\th)$&\eqref{BC3} \\ \hline
\multirow{2}{*}{$\u\cdot \n=f_0(\th),\,p=p_0(\th)$} & 
$\u\cdot \n=f_1(\th)$&\eqref{BC4} \\ \cline{2-3} 
& $p=p_1(\th)$&\eqref{BC5}\\ \hline
\end{tabular}
\vspace{8pt}
\caption{Boundary Conditions}
\label{tab_BC}
\end{table}

\vspace{-20pt}	

Roughly speaking, the core idea of the vorticity transport method lies in verifying the iterative scheme $\hat \u\to\om\to\u$. This theoretical framework was rigorously developed by Albert \cite{Alber1992}, who proved the existence of non-vanishing vorticity solutions in Sobolev space perturbed from irrational base flows for the 3D steady Euler equation. This vorticity transport framework offers significant flexibility in treating various boundary conditions. Extending this approach, the authors in \cite{TX2009} constructed solutions to the 3D steady Euler equation with a class of additional boundary conditions for the vorticity. Our analysis reveals that the vorticity transport method can systematically resolve each boundary condition listed in Table \ref{tab_BC}.

Taking the curl of the first equation in \eqref{InE_eq}, we obtain the governing equation for the vorticity $\om$:
\begin{equation}\label{eq_trans_om}
(\u\cdot\nabla)\om=0.
\end{equation}

For a given velocity field $\hat\u$ in an appropriate function space, the vorticity $\om$ can be uniquely determined through the boundary value problem:
\begin{equation*}
(\hat\u\cdot\nabla)\om=0,\,
\om(r_0,\th)=\om_0(\th).
\end{equation*}
The corresponding velocity field $\u$ is then obtained as the unique solution to the following div-curl system with normal boundary conditions:
\begin{equation*}
\nabla\times \u=\om,\,
\nabla\cdot \u=0,\,
(\u\cdot \n)(r_i,\th)=f_i(\th),\,i=0,1.
\end{equation*}	

The key point is whether the boundary conditions in Table \ref{tab_BC} exactly determine $\om_0(\th)$ and $f_1(\th)$. It turns out that the boundary conditions specified in the first row of Table \ref{tab_BC} for the inner circle directly determine the initial vorticity $\om_0$ (see \eqref{bc1*_om_0}), whereas those in the second row establish $\om_0$ through an iterative process (see \eqref{bc4_om0}). All boundary conditions enumerated in Table \ref{tab_BC} permit either direct or indirect determination of $f_1(\th)$ (cf. \eqref{bc5_def_f1} and \eqref{bc2*_def_f1}), thereby enabling resolution of all cases presented in the table. See Section \ref{sec3} for detailed analysis. 

\begin{remark}
Both the Grad-Shafranov method and the vorticity transport method present  difficulties when investigating the well-posedness of the Euler equation \eqref{InE_Pol} subject to the following boundary conditions:
\begin{equation*}
(\u\cdot \n)(r_0,\th)=f_0(\th),\,p(r_0,\th)=p_0(\th),\,
(\frac{1}{2}|\u|^2+p)(r_1,\th)=b_1(\th).
\end{equation*}
The Grad-Shafranov method becomes inapplicable due to the lack of Bernoulli function boundary data at the inner circle. For the vorticity transport method, the fundamental obstacle stems from constructing a globally valid $f_1(\th)$. Although the vorticity $\om(r,\th)$ can be solved analogously to the case of \eqref{BC5} (see \eqref{bc5_eq_om} and \eqref{bc5_om0}), the formula
\begin{equation*}
f_1(\th)=\frac{\p_\th b_1(\th)}{\om(r_1,\th)}
\end{equation*}
may not yield a well-defined function for all $\th\in\mb{T}_{2\pi}$. 
\end{remark}

\begin{remark}
It is worth mentioning that the authors  of \cite{AV2024} consider the following boundary conditions
\begin{equation*}
\u(r,\th)=\u_0(\th),\,(\u\cdot\n)(r_1,\th)=f_1(\th),
\end{equation*} 
which are not covered by Table \ref{tab_BC}. Addressing such a boundary value problem seems quite difficult, particularly in the designation of the iteration scheme and verification of the tangential boundary condition.
\end{remark}

\begin{remark}
All theorems in this paper concerning the well-posedness of the steady Euler equation \eqref{InE_Pol} remain valid when the boundary data on the inner and outer circles in Table \ref{tab_BC} are interchanged.
\end{remark}

\begin{remark}
The equivalence between the steady Euler system \eqref{InE_eq} and the MHS equation \eqref{eq_MHS} implies that all boundary conditions in Table \ref{tab_BC} admit equivalent formulations in terms of $(\mf{B},p)$, and inherit corresponding well-posedness results, though detailed presentations are omitted for brevity.
\end{remark}

\subsection{Notation} We employ the following notation throughout this paper.
\begin{itemize}
\item  $\mb{T}_{2\pi}$ denotes the one-dimensional torus with period $2\pi$.
\item Let $C_b(\Om)$ be the set of bounded continuous functions on $\Om$. The H{\"o}lder space $C^{k,\alpha}(\Omega)$, where $k=0,1,2\cdots$ and $\alpha \in (0,1)$, consists of functions $f \in C_b(\Omega)$ with finite norm:
\begin{equation*}
\|f\|_{C^{k,\alpha}} := \|f\|_{C^k} + \sup_{\substack{x,y \in \Omega \\ x \neq y}} \frac{|\nabla^k f(x) - \nabla^k f(y)|}{|x-y|^\alpha}.
\end{equation*}

\item A function $f$ is called Lipschitz continuous if
\begin{equation*}
|f(x)-f(y)|\leq L|x-y|
\end{equation*}
for some positive constant $L$ and all $x,y\in \Om$. We write
\begin{equation*}
\|f\|_{Lip}:=\sup_{\substack{{x,y\in\Om}\\{x\neq y}}}\frac{|f(x)-f(y)|}{|x-y|}.
\end{equation*} 
We denote the space $C^{0,1}$ as the continuous functions with finite norm 
\begin{equation*}
\|f\|_{C^{0,1}}:=\|f\|_{C^0}+\|f\|_{Lip}.
\end{equation*}

\item We identify the space $C^{k,\al}(\mb{T}_{2\pi})$ with $C^{k,\al}(\R)$ such that $f(s+2\pi)=f(s)$. 

\item A function $g(r,\th)\in H^1_{per}(\Om)$ means that $g(r,\th)$ belongs to the Sobolev space $H^1(\Om)$ and $g(r,\eta+\th)=g(r,\th)$.

\item For any $f\in C^{k,\al}(\mb{T}_{2\pi})$, we denote by $f^{ave}$ its mean value over one period, namely,
\begin{equation*}
f^{ave}=\frac{1}{2\pi}\int_0^{2\pi}f(s)\,ds.
\end{equation*}

\item  Consider a vector field $\v(r,\th)=v_r(r,\th)\e_r+v_\th(r,\th)\e_\th$ in cylindrical coordinates. We define the following operator:
\begin{equation*}
\begin{split}
\div_{c}\v&:=\frac{1}{r}\p_r(r v_r)+\frac{1}{r}\p_\th v_\th,\\
\curl_c&\v:=\frac{1}{r}\p_r(r v_\th)-\frac{1}{r}\p_\th v_r,\\
\nabla_c&:=\p_r\e_r+\frac{1}{r}\p_\th\e_\th.
\end{split}
\end{equation*}

\end{itemize}

\subsection{Paper Organization}
This paper is structured as follows. In Section \ref{sec2}, we examine three classes of boundary conditions, \eqref{BC1}, \eqref{BC2}, and \eqref{BC3}, for the Euler equation \eqref{InE_Pol} using the Grad-Shafranov approach. We demonstrate that the steady Euler equation \eqref{InE_Pol}, when endowed with any of these boundary conditions, can be reformulated as a boundary value problem for a nonlinear elliptic equation through the introduction of a stream function. Two additional boundary conditions, \eqref{BC4} and \eqref{BC5}, are investigated in Sections \ref{sec3.1} and \ref{sec3.2}, respectively, employing the vorticity transport method. In the last subsection \ref{sec3.3}, we provide an alternative proof for the boundary value problem originally formulated in section \ref{sec2} by applying the method developed in section \ref{sec3}.

\section{Grad-Shafranov approach}\label{sec2}
In this section, we investigate the well-posedness of the Euler equations \eqref{InE_Pol} under three different types of boundary conditions \eqref{BC1},\eqref{BC2}, and \eqref{BC3} by the Grad-Shafranov approach.

The first type of boundary condition with which we are concerned is as follows.
\begin{equation}\tag{BC1}\label{BC1}
\begin{cases}
r_0u_r(r_0,\th)=f_0(\th),\\
r_1u_r(r_1,\th)=f_1(\th),\\
\frac{1}{2}|\u(r_0,\th)|^2+p(r_0,\th)=b_0(\th),\\
\int_{r_0}^{r_1} u_\th(r,0)\,dr=j_0,
\end{cases}
\end{equation}
where $f_i(\th),\,i=0,1$ and $b_0(\th)$ are given $2\pi$-periodic functions, and 
\begin{equation}\label{flux_rela}
\int_{0}^{2\pi}f_1(s)\,ds=\int_{0}^{2\pi}f_0(s)\,ds:=J_0.
\end{equation}
\begin{theorem}\label{thm1}
Let $f_0(\th),f_1(\th),b_0(\th)\in C^{1,\al}(\mb{T}_{2\pi})$ and $f_0(\th)\geq\de_0>0$ for all $\th\in\mb{T}_{2\pi}$. Then the steady incompressible Euler equation \eqref{InE_Pol} admits a solution $(\u,p)\in (C^{1,\al}(\Om))^3$ satisfying the boundary condition \eqref{BC1}. Moreover,  the solution $(\u,p)$ is unique when $b_0(\th)\in C^{1,1}(\mb{T}_{2\pi})$, and $C^{0,1}$-norm of $b_0'$ is sufficiently small. 
\end{theorem}

\begin{remark}
The existence of solutions is established without any smallness assumption on the boundary data. For uniqueness, only the condition that $\|b_0'\|_{C^0}+\|b_0'\|_{Lip}$ be sufficiently small is required, which permits $b_0$ itself to be large.
\end{remark}

The second class of boundary conditions is derived by substituting the radial velocity condition at $r=r_1$ in \eqref{BC1} with the value of the Bernoulli function. Specifically, we impose:
\begin{equation}\tag{BC2}\label{BC2}
\begin{cases}
r_0u_r(r_0,\th)=f_0(\th),\\
\frac{1}{2}|\u(r_0,\th)|^2+p(r_0,\th)=b_0(\th),\\  
\frac{1}{2}|\u(r_0,\th)|^2+p(r_0,\th)=b_1(\th),
\end{cases}
\end{equation}
where $f_0(\th),b_0(\th),b_1(\th)$ are given $2\pi$-periodic functions. Since the Bernoulli function satisfies a transport equation, its boundary values cannot be independently prescribed on both the inner and outer circles. It is mathematically consistent and physically admissible to assume that the Bernoulli function's boundary values on the inner circle are diffeomorphically equivalent to those on the outer circle. Hence, we suppose that
\begin{equation}
b_1(\th)=(b_0\circ T)(\th),
\end{equation}
where $T:\mb{T}_{2\pi}\to \mb{T}_{2\pi}$ is a given orientation-preserving  $C^{2,\al}$ diffeomorphism.

\begin{theorem}\label{thm2}
Let $f_0(\th),b_0(\th)\in C^{1,\al}(\mb{T}_{2\pi})$ and $f_0(\th)\geq\de_0>0$ for all $\th\in\mb{T}_{2\pi}$. Then the steady incompressible Euler equation \eqref{InE_Pol} admits a solution $(\u,p)\in (C^{1,\al}(\Om))^3$ satisfying the boundary condition \eqref{BC2}. Moreover, the solution $(\u,p)$ is unique when $b_0(\th)\in C^{1,1}(\mb{T}_{2\pi})$, and $C^{0,1}$-norm of $b_0'$ is sufficiently small. 
\end{theorem}

Next, we study the modified boundary value problem where the pressure $p$ is specified through a Dirichlet condition at $r=r_1$, replacing the value of the Bernoulli function $\frac{1}{2}\u^2+p$ prescribed at this boundary. This modification leads to a nonlinear boundary condition at $r=r_1$, which is fundamentally different from the linear framework induced by \eqref{BC2}. To this end, it is natural to investigate the well-posedness of system \eqref{InE_Pol} within the perturbation framework around the equilibrium state $(\bar \u,\bar p)=(\frac{1}{r}\e_r,-\frac{1}{2 r^2})$. We adopt the following boundary conditions
\begin{equation}\tag{BC3}\label{BC3}
\begin{cases}
r_0u_r(r_0,\th)=1+f_0(\th),\\
\frac{1}{2}|\u(r_0,\th)|^2+p(r_0,\th)=b_0(\th),\\
\int_{r_0}^{r_1} u_\th(r,0)\,dr=j_0,\\
\p_\th p(r_1,\th)=\p_{\th}p_1(\th).
\end{cases}
\end{equation}
\begin{theorem}\label{thm3}
Given $f_0(\th),b_0(\th),p_1(\th)\in C^{1,\al}(\mb{T}_{2\pi})$ and the constant $j_0\in\R$. There exists a small constant $\vep_0>0$ such that for all $\vep\in(0,\vep_0]$, if
\begin{equation*}
\|f_0\|_{C^{1,\al}}+\|b_0\|_{C^{1,\al}}+\|p_1'\|_{C^{\al}}+|j_0|\leq \vep,
\end{equation*}
then the steady incompressible Euler equation \eqref{InE_Pol} with the boundary condition \eqref{BC3} has a unique solution  $(\u,p)\in (C^{1,\al}(\Om))^3$.
\end{theorem}

By introducing an appropriate stream function, the Grad-Shafranov approach reduces the well-posedness analysis of the Euler system \eqref{InE_Pol}, when equipped with any of the boundary conditions \eqref{BC1}, \eqref{BC2}, or \eqref{BC3}, to the study of a simpler elliptic equation. Utilizing the incompressibility condition in \eqref{InE_Pol}, we know that there is a stream function $\phi(z_1,z_2)$ with $(r,\th)=(z_1,z_2)$, such that
\begin{equation}\label{def_strm_fun}
\p_{z_1}\phi=- u_\th(z),\,\p_{z_2}\phi=z_1 u_r(z).
\end{equation}

Denoting
\begin{equation*}
B=\frac{1}{2}|\u|^2+p,\,\om=\frac{1}{r}\big( \p_r(r u_\th)-\p_\th u_r \big),
\end{equation*}
then it follows from \eqref{InE_Pol} that 
\begin{equation}\label{eq_phi_B_om}
\begin{cases}
\nabla_z^\perp\phi\cdot \nabla_z B=0,\\
\om\nabla_z\phi+\nabla_z B=0,
\end{cases}
\end{equation}
where $\nabla^\perp_z=(-\p_{z_2},\p_{z_1})$. The first equation in \eqref{eq_phi_B_om} suggests $B=B(\phi)$ and the second equation becomes $\om+B'(\phi)=0$. Since $\om$ can be rewritten as
\begin{equation*}
\begin{split}
\om =-\frac{1}{z_1}\big(\p_{z_1}(z_1\p_{z_1}\phi)+\p_{z_2} (\frac{1}{z_1}\p_{z_2}\phi) \big),
\end{split}
\end{equation*}
we finally get 
\begin{equation}\label{eq_phi}
-\div_z(K(z_1)\nabla_z\phi)=z_1B'(\phi),
\end{equation}
where 
\begin{equation*}
K_{11}(z_1)=z_1,\,K_{22}=\frac{1}{z_1},\,K_{12}=K_{21}=0.
\end{equation*}

We will demonstrate that \eqref{eq_phi} constitutes a closed system for the unknown function $\phi$ when equipped with one of the boundary conditions \eqref{BC1},\eqref{BC2}, and \eqref{BC3}. This reduction is available because the Bernoulli function $B(\cdot)$ is determined by boundary data. Thus, the analysis of steady Euler equations is effectively reduced to the study of the boundary value problem for the nonlinear elliptic equation \eqref{eq_phi}. In the rest of this section, we derive the boundary data for the stream function $\phi$ in each case and establish the well-posedness of $\phi$ via two different methods.

\subsection{variational method} In this subsection, we study the boundary value problem for the stream function $\phi$ when considering boundary conditions \eqref{BC1} or \eqref{BC2}.

To this end, denoting 
\begin{equation}\label{phi0}
\phi_0(z_2)=\int_{0}^{z_2}f_0(s)\,ds,
\end{equation}
then the first relation in \eqref{BC1} implies
\begin{equation}
\frac{d \phi_0(z_2)}{d z_2}=r_0 v_r(r_0,z_2)=f_0(z_2),
\end{equation}
Thanks to $f_0(z_2)>0$, we know that $\phi_0(z_2)$ is invertible. Specifically, there is a function $Z(\tau)$ such that
\begin{equation*}
Z(\tau)=\phi_0^{-1}(\tau),\,\tau\in\R.
\end{equation*}
Moreover, we have
\begin{equation*}
\phi_0(z_2+2\pi)-\phi_0(z_2)=\int_{z_2}^{z_2+2\pi}f_0(s)ds\equiv J_0,
\end{equation*}
and
\begin{equation*}
Z(\tau+J_0)-Z(\tau)=2\pi,
\end{equation*}
where $J_0$ is defined in \eqref{flux_rela}.

The second line in \eqref{BC1} shows that
\begin{equation*}
B(r_0,z_2)=b_0(z_2).
\end{equation*}

Therefore, we can define the single-variable functions $B(\cdot)$ as
\begin{equation}\label{def_B_BC12}
B(\tau)=b_0(Z(\tau)),
\end{equation}
and it's easy to check that
\begin{equation*}
B(\tau+J_0)=B(\tau).
\end{equation*}

In the case of prescribing the boundary condition \eqref{BC1}, we define $\phi_1(z_2)$ as
\begin{equation}\label{phi1_1}
\phi_1(z_2)=-j_0+\int_{0}^{z_2}f_1(s)\,ds.
\end{equation}

When the boundary condition \eqref{BC2} is imposed, the function $\phi_1(z_2)$ is constructed through the diffeomorphism $T$ via the composition
\begin{equation}\label{phi1_2}
\phi_1(z_2)=\phi_0\circ T(z_2).
\end{equation}

Moreover, a direct computation shows that $\phi_1(z_2)$ constructed in both \eqref{phi1_1} and \eqref{phi1_2} satisfies
\begin{equation*}
\begin{split}
\phi_1(z_2+2\pi)-\phi_1(z_2)=J_0.
\end{split} 
\end{equation*}

To sum up, we conclude that the stream function $\phi(z)$ is governed by the boundary value problem below
\begin{equation}\label{eq_phi_BC12}
\begin{cases}
-\div_z(K(z_1)\nabla_z\phi)=z_1B'(\phi),\,z\in \Om,\\
\phi(r_0,z_2)=\phi_0(z_2), \\
\phi(r_1,z_2)=\phi_1(z_2),\\
\phi(z_1,z_2+2\pi)=\phi(z_1,z_2)+J_0,\,z_1\in (r_0,r_1).
\end{cases}
\end{equation}

\begin{proposition}\label{prop1}
Let $B(\phi)$ be determined by \eqref{def_B_BC12}. Under the assumptions of either Theorem \ref{thm1} or Theorem \ref{thm2}, there is a solution $\phi(z)\in C^{2,\al}(\Om)$ to \eqref{eq_phi_BC12}. Moreover, if the $C^{1,\al}$-norms of $f_0,b_0$ are small enough, the solution $\phi(z)$ to \eqref{eq_phi_BC12} is unique. 
\end{proposition}
\begin{proof}
Setting
\begin{equation}\label{phi_psi_1}
\psi(z)=\phi(z)-\bar{J}_0 z_2,\,\bar{J}_0=\frac{J_0}{2\pi},
\end{equation}
then it follows from \eqref{eq_phi_BC12} that $\psi(z)$ satisfies
\begin{equation}\label{eq_psi_1}
\begin{cases}
-\div_z(K(z_1)\nabla_z \psi(z))=z_1 B'(\psi+\bar{J}_0 z_2),\,z\in \Om,\\
\psi(r_0,z_2)=\psi_0(z_2):=\phi_0(z_2)-\bar{J}_0 z_2,\,z_2\in \Ga_0, \\
\psi(r_1,z_2)=\psi_1(z_2):=\phi_1(z_2)-\bar{J}_0 z_2,\,z_2\in \Ga_1,\\
\end{cases}
\end{equation}
and the periodic condition:
\begin{equation}\label{peri_psi_1}
\psi(z_1,z_2+2\pi)=\psi(z_1,z_2),\forall z_1\in(r_0,r_1).
\end{equation}
We introduce the following functional
\begin{equation*}
\ml{I}[\psi]:=\int_\Om\,\L(\nabla\psi,\psi,z)\,dz
\end{equation*}
on the admissible space:
\begin{equation*}
\ml{A}:=\{\psi\in H_{per}^1(\Om):\psi(r_0,z_2)=\psi_0,\,\psi(r_1,z_2)=\psi_1\},
\end{equation*}
where
\begin{equation*}
\L(\nabla\psi,\psi,z):=\frac{1}{2}(K(z_1)\nabla\psi)\cdot\nabla\psi-
z_1B(\psi+\bar{J}_0z_2).
\end{equation*}
It is evident that $\L(\nabla\psi,\psi,z)$ is coercive and convex in the first variable. As a result, we derive that there is at least one minimizer $\psi$ so that
\begin{equation*}
\ml{I}[\psi]=\min_{\vphi\in \ml{A}}\ml{I}[\vphi],
\end{equation*}
and $\psi$ is exactly a weak solution to \eqref{eq_psi_1} satisfying periodicity \eqref{peri_psi_1}(cf. \cite{Evans}). Moreover, since $z_1B'(\psi+\bar{J}_0z_2)\in C^{\al}(\Om),\psi_0,\psi_1\in C^{2,\al}(\Om)$, we get $\psi\in C^{2,\al}(\Om)$ by the standard theory \cite{GT2001} of elliptic regularity.

To show the uniqueness, let $\psi^i\in C^{2,\al}(\Om),i=1,2$ are two solutions to \eqref{eq_psi_1}, then the difference $\tilde\psi=\psi^1-\psi^2$ satisfies
\begin{equation}\label{eq_tile_phi}
\begin{cases}
-\div_z(K(z_1)\nabla_z \tilde\psi(z))=z_1B'(\psi^1+\bar{J}_0 z_2)-z_1B'(\psi^2+\bar{J}_0 z_2),\\
\tilde\psi(r_0,z_2)=0,\\
\tilde\psi(r_1,z_2)=0,
\end{cases}
\end{equation}
By the definition of the Bernoulli function, namely \eqref{def_B_BC12}, we get
\begin{equation}\label{diff_berno}
\begin{split}
|z_1&B'(\phi^1)-z_1B'(\phi^2)|=|z_1b_0'(Z(\phi^1))Z'(\phi^1)-z_1b_0'(Z(\phi^2))Z'(\phi^2)|\\
&\leq z_1|b_0'(Z(\phi^1))-b_0'(Z(\phi^2))||Z'(\phi^1)|
+z_1|b_0'(Z(\phi^2))||Z'(\phi^1)-Z'(\phi^2)|\\
&\leq r_1\big( \|b_0'\|_{Lip}\|Z'\|_{C^0}^2+\|b_0'\|_{C^0}\|Z''\|_{C^0} \big)
|\phi^1-\phi^2|\\
&\leq r_1\big( \frac{1}{\de_0^2}\|b_0'\|_{Lip}+\frac{1}{\de_0^3}\|b_0'\|_{C^0}\|f_0'\|_{C^0} \big)
|\phi^1-\phi^2|\\
&\leq C(r_1,\de_0,\|f_0'\|_{C^0})\|b_0'\|_{C^{0,1}}|\phi^1-\phi^2|,
\end{split} 
\end{equation}
where we have used the fact that
\begin{equation*}
Z'=\frac{1}{f_0},\,Z''=-\frac{f_0'}{f_0^3}.
\end{equation*}

It follows from \eqref{eq_tile_phi} that
\begin{equation*}
\|\tilde\psi\|_{C^{2,\al}}\leq C\|b_0'\|_{C^{0,1}}\|\tilde\psi\|_{C^{\al}}.
\end{equation*}

If the $C^{0,1}$-norm of $b_0'$ is sufficiently small, we obtain $\tilde\psi=0$, i.e. $\psi^1=\psi^2$. 
\end{proof}

The results of Theorem \ref{thm1} and Theorem \ref{thm2} are essentially direct consequences of Proposition \ref{prop1}. For brevity, we present here only the detailed proof of Theorem \ref{thm2}.

{\bf Proof of Theorem \ref{thm2}.}
Once the stream function $\phi(z)\in C^{2,\al}$ is constructed by Proposition \ref{prop1}, the solution $(\u,p)$ to the steady Euler equation \eqref{InE_Pol} are sequentially derived by
\begin{equation*}
u_r(z)=\frac{1}{z_1}\p_{z_2}\phi(z)\in C^{1,\al},\,u_\th(z)=-\p_{z_1}\phi(z)\in C^{1,\al},
\end{equation*}
and 
\begin{equation*}
p=B(\phi)-\frac{1}{2}|\u|^2\in C^{1,\al}.
\end{equation*}
It remains to show that $(\u,p)$ satisfies the boundary condition \eqref{BC1}. For the radial velocity, we have
\begin{equation*}
\begin{split}
r_1u_r(r_0,\th)=\p_\th\phi(r_0,\th)=\p_{\th}\phi_0(\th)=f_0(\th).
\end{split}
\end{equation*}

Recalling the definition of $\phi_1(\th)$ in \eqref{phi1_2}, we conclude that $\phi_1(\th)$ is also invertible due to 
\begin{equation*}
\phi_1'(\th)=f_0(T(\th))T'(\th)>0
\end{equation*}
for all $\th\in(0,2\pi]$. Then, by the construction of Bernoulli function $B$ in \eqref{def_B_BC12}, we obtain
\begin{equation}\label{check_B}
\begin{split}
B(\phi)&=b_0(\phi_0^{-1}(\phi))
=(b_0\circ T\circ T^{-1}\circ\phi_0^{-1})(\phi)=(b_0\circ T)\circ(\phi_0\circ T)^{-1}(\phi)\\
&=b_1(\phi_1^{-1}(\phi)).
\end{split}
\end{equation}
Consequently, we have
\begin{equation*}
(\frac{1}{2}|\u|^2+p)(r_i,\th)=B(\phi_i(\th))=(b_i\circ\phi_i^{-1}\circ\phi_i)(\th)=b_i(\th),\,i=0,1.
\end{equation*}
The proof is completed. \qed

\subsection{Iteration method}
Next, let's focus on the boundary condition \eqref{BC3}. Defining 
\begin{equation}
\phi_0(z_2)=z_2+\int_{0}^{z_2}f_0(s)\,ds,
\end{equation}
and noting
\begin{equation*}
\frac{d\phi_0(z_2)}{d z_2}=1+f_0(z_2)>0,
\end{equation*}
we conclude that there exists a function $Z(\tau)$ such that
\begin{equation*}
Z(\tau)=\phi_0^{-1}(\tau),\,\tau\in\R.
\end{equation*}
It's obvious that
\begin{equation*}
\phi_0(z_2+2\pi)-\phi_0(z_2)=J_1,\,
Z(\tau+J_1)-Z(\tau)=2\pi,
\end{equation*}
where $J_1=J_0+2\pi$.

Similarly, the second relation in \eqref{BC3} implies that $A B(\cdot)$ are given by
\begin{equation}\label{def_B_BC3}
A(\tau)=a_0(Z(\tau)),\,B(\tau)=b_0(Z(\tau)),
\end{equation}
which satisfy
\begin{equation*}
B(\tau+J_1)=B(\tau).
\end{equation*}

As a result, $\phi(z)$ solves the following boundary value problem:
\begin{equation}\label{eq_phi_BC3}
\begin{cases}
-\div_z(K(z_1)\nabla_z \phi(z))=z_1 B'(\phi),\\
\phi(r_0,z_2)=\phi_0(z_2), \\
-B'(\phi)\p_{z_2}\phi+
\frac{1}{2}\p_{z_2}(\p_{z_1}\phi)^2+\frac{1}{2 r^2}\p_{z_2}(\p_{z_2}\phi)^2\bigg|_{r=r_1}=-\p_{z_2} p_1
\\
\phi(r_1,0)=-j_0\\
\phi(z_1,z_2+2\pi\ka)=\phi(z_1,z_2)+J_1,\,z_1\in (r_0,r_1).
\end{cases}
\end{equation}

Prior to analyzing the nonlinear boundary value problem \eqref{eq_phi_BC3}, we establish the following lemma regarding the well-posedness of the corresponding linearized system.

\begin{lemma}\label{lem_lin_eli}
Assume that $F(z)\in C^{\al}(\Om),f_i(\th)\in C^{2,\al}(\mb{T}_{2\pi}),\, i=1,2$. If 
\begin{equation}\label{f0_f1_BC3}
\int_{0}^{2\pi}\,f_1(s)\,ds=\int_{0}^{2\pi}\,f_0(s)\,ds:=J_0,
\end{equation}
then the following boundary value problem
\begin{equation}\label{lins_sys}
\begin{cases}
-\div_z(K(z_1)\nabla_z \phi(z))=F(z),\\
\phi(r_0,z_2)=\int_0^{z_2}f_0(s)\,ds, \\
\phi(r_1,z_2)=-j_0+\int_0^{z_2}f_1(s)\,ds,\\
\phi(z_1,z_2+2\pi)=\phi(z_1,z_2)+J_0,\,z_1\in (r_0,r_1),
\end{cases}
\end{equation}
admits a unique solution $\psi\in C^{2,\al}(\Om)$ to \eqref{lins_sys} satisfying
\begin{equation}\label{shau_est}
\|\psi\|_{C^{2,\al}(\Om)}\leq C\big(|j_0|+\|F\|_{C^{\al}}
+\|f_0\|_{C^{1,\al}}+\|f_1\|_{C^{1,\al}}\big).
\end{equation}
for some positive constant $C>0$.
\end{lemma}

The result of Lemma \ref{lem_lin_eli} follows immediately from the standard theory \cite{GT2001} of elliptic equation via the transformation $\phi=\psi+\bar J_0 z_2$, where $\bar J_0=\frac{J_0}{2\pi}$. We now proceed to establish the well-posedness of the nonlinear problem \eqref{eq_phi_BC3}.

\begin{proposition}\label{prop2}
Let $B(\phi)$ be given by \eqref{def_B_BC3}. Under the same assumptions as Theorem \ref{thm2}, there exists a unique solution $\phi(z)\in C^{2,\al}(\Om)$ to \eqref{eq_phi_BC3}.  
\end{proposition}
\begin{proof}
Letting
\begin{equation}\label{phi_vphi}
\phi=\bar\phi+\vphi,\,\bar\phi=z_2,
\end{equation}
and substituting it into \eqref{eq_phi_BC3}, we obtain
\begin{equation}\label{eq_vphi_bd2}
\begin{cases}
-\div_z(K(z_1)\nabla_z \vphi(z))=z_1b_0'(z_2)+F_R(\vphi),\,z\in\Om,\\
\vphi(r_0,z_2)=\int_0^{z_2}f_0(s)\,ds, \\
\vphi(r_1,z_2)=-j_0+\int_0^{z_2}f_1(s;\vphi)\,ds,\\
\vphi(z_1,2\pi)=\vphi(z_1,0)+J_0,\,z_1\in (r_0,r_1),
\end{cases}
\end{equation}
where $F_R(\vphi)$ is the remainder term of second order and $f_1(z_2)$ is given by
\begin{equation}\label{BC3_def_f1}
f_1(z_2)=f_1(0)+\int_{0}^{z_2}F_1(s;\vphi,f_1)-F_1^{ave}\,ds,
\end{equation}
and
\begin{equation}
F_1(z_2;\vphi,f_1):=\bigg[r_1^2b_0'(Z(\phi))Z'(\phi)-\frac{r_1^2}{2(1+f_1)}\p_{z_2}(\p_{z_1}\vphi)^2\bigg]_{r=r_1}-\frac{r_1^2}{1+f_1}\p_{z_2}p_1.
\end{equation}

The constant $f_1(0;\vphi,f_1)$ in \eqref{BC3_def_f1} is determined by 
\begin{equation*}
f_1(0)=\frac{1}{2\pi}\big( J_0+\int_{0}^{2\pi}
\big(F_1(z_2;\vphi,f_1)-F_1^{ave}\big)
(2\pi-z_2)\,dz_2 \big),
\end{equation*}
to ensure the compatibility condition \eqref{f0_f1_BC3} is satisfied.

The well-posedness of \eqref{eq_vphi_bd2} can be established by iteration. To this end, we introduce the Banach space
\begin{equation*}
\Phi_\delta:=\{(\vphi,f_1)\in C^{2,\al}(\Om)\times C^{1,\al}(\mb{T}_{2\pi}):\|\vphi\|_{C^{2,\al}}+\|f_1\|_{C^{1,\al}}\leq \delta\}.
\end{equation*}
Here $\delta>0$ is a small constant to be determined later. For given $(\hat\vphi,\hat f_1)\in\Phi_\delta,$,  we first define $f_1(z_2)\in C^{1,\al}$ as
\begin{equation}
f_1(z_2)=f_1(0)+\int_{0}^{z_2}F_1(s;\hat\vphi,\hat{f}_1)-F_1^{ave}\,ds.
\end{equation}

A direct computation shows that
\begin{equation}\label{est_BC3_f1}
\begin{split}
\|f_1\|_{C^{1,\al}}&\leq |f_1(0)|+C\|F_1(s;\hat\vphi,\hat{f}_1)-F_1^{ave}\|_{C^{\al}}\\
&\leq C\bigg(
|J_0|+\|b_0'\|_{C^{\al}}\|Z'\|_{C^{\al}}+(1+\|\hat{f}_1\|_{C^{\al}})
(\|\hat\vphi\|^2_{C^{2,\al}}+\|\p_{z_2}p_1\|_{C^{\al}})
\bigg)\\
&\leq C(\vep+\vep\de+\de^2).
\end{split}
\end{equation}

With the help of $f_1$ defined above, we conclude that there is a unique solution $\vphi\in C^{2,\al}(\Om)$ to the following linear boundary value problem:
\begin{equation}
\begin{cases}
-\div_z(K(z_1)\nabla_z \vphi(z))=z_1b_0'(z_2)+F_R(\hat\vphi),\\
\vphi(r_0,z_2)=\int_0^{z_2}f_0(s)\,ds, \\
\vphi(r_1,z_2)=-j_0+\int_0^{z_2}f_1(s)\,ds,\\
\vphi(z_1,z_2+2\pi)=\vphi(z_1,z_2)+J_0,\,z_1\in (r_0,r_1),
\end{cases}
\end{equation}
Applying lemma \ref{lem_lin_eli} and using \eqref{est_BC3_f1}, we have
\begin{equation}
\begin{split}
\|\vphi\|_{C^{2,\al}}+\|f_1\|_{C^{1,\al}}\leq C(\vep +\vep\de+\de^2)
\end{split}
\end{equation}

Selecting $\de=3C\vep$ and $\vep<\min\{1,\frac{1}{3C},\frac{1}{9C^2}\}$, we can define $\ml{T}:\Phi_\delta \to\Phi_\delta$ as
\begin{equation*}
\ml{T}(\hat\vphi,\hat f_1)=(\vphi,f_1).
\end{equation*}

We claim that $\ml{T}$ is also contractible. Indeed, for any two given vectors $(\hat\vphi^i,\hat f_1^i)\in\Phi_\delta,\,i=1,2$, the operator $\ml{T}$ generates corresponding image vectors $(\vphi^i,f_1^i),\,i=1,2$. 

By the definition \eqref{BC3_def_f1}, we obtain the following estimate for the difference
\begin{equation}\label{est_BC3_diff_f1}
\begin{split}
\|f^1_1-&f^2_1\|_{C^{1,\al}}\leq |f^1_1(0)-f^2_1(0)|+C\|F_1(s;\hat\vphi^1,\hat{f}_1^1)-F_1(s;\hat\vphi^2,\hat{f}_1^2)\|_{C^{\al}}\\
&\leq C\|b_0'(Z(\hat\phi^1))Z'(\hat\phi^1)-b_0'(Z(\hat\phi^2))Z'(\hat\phi^2)\|_{C^\al}\\
&+C\bigg\|\frac{\p_{z_2}(\p_{z_1}\hat\vphi^1)^2}{1+\hat f_1^1}-\frac{\p_{z_2}(\p_{z_1}\hat\vphi^2)^2}{1+\hat f_1^2}\bigg\|_{C^\al}
+C\bigg\|\frac{\p_{z_2}p_1}{1+\hat f_1^1}-\frac{\p_{z_2}p_1}{1+\hat f_1^2}\bigg\|_{C^\al}\\
&\leq C\vep\|\hat\vphi^1-\hat\vphi^2\|_{C^{2,\al}}+\vep\|\hat f^1_1-\hat f^2_2\|_{C^{\al}},
\end{split}
\end{equation}

where $\hat\phi^i=z_2+\hat\vphi^i,\,i=1,2$, and we have invoked the estimate \eqref{diff_berno}.

Utilizing \eqref{shau_est} again to derive
\begin{equation}\label{BC3_est_diff_vphi}
\begin{split}
\|\vphi^1-\vphi^2\|_{C^{2,\al}}&\leq C\big(
\|b_0(\hat\phi^1)Z(\hat\phi^1)-b_0(\hat\phi^2)Z(\hat\phi^2)\|_{C^{\al}}
+\|f_1^1-f_1^2\|_{C^{1,\al}}\big)\\
&\leq C\vep \|\hat \vphi^1-\hat \vphi^2\|_{C^{2,\al}}
+C\vep\|\hat f^1_1-\hat f^2_2\|_{C^{\al}} .
\end{split}
\end{equation}

By combining estimates \eqref{est_BC3_diff_f1} and \eqref{BC3_est_diff_vphi}, we establish the existence of a small constant $\vep_0>0$ such that for all $\vep\in(0,\vep_0]$, the operator $\ml{T}$ is a contraction. This implies that there exists a unique solution $\vphi(z)\in C^{2,\al}(\Om)$ to the nonlinear boundary value problem \eqref{eq_vphi_bd2}. Proposition \eqref{prop1} then follows directly from the relation \eqref{phi_vphi}.
\end{proof}

The verification of boundary conditions in Theorem \ref{thm3} will be addressed in the final subsection \ref{sec3.3}, where we provide an alternative proof of this result. We highlight that the boundary value problem may become ill-posed if the outer circle condition in \eqref{BC3} is changed from specifying the derivative of pressure to specifying the pressure itself. Let us define the boundary operator $\ga_1^{BC3}$ acting on the pressure $p$ obtained from Theorem \ref{thm3} as
\begin{equation*}
(\ga^{BC3}_1p)(\th)=p(r_1,\th).
\end{equation*}
With this notation, we have the following result.

\begin{thmp}\label{thm3'}
Let $f_0(\th),b_0(\th),p_1(\th)\in C^{1,\al}(\mb{T}_{2\pi})$ and the constant $j_0\in\R$. There exists a small constant $\vep_0>0$ such that for all $\vep\in(0,\vep_0]$, if
\begin{equation*}
\|f_0\|_{C^{1,\al}}+\|b_0\|_{C^{1,\al}}+\|p_1\|_{C^{1,\al}}+|j_0|\leq \vep,
\end{equation*}
then the steady incompressible Euler equation \eqref{InE_Pol} equipped with boundary conditions
\begin{equation}\tag{BC3$'$}\label{BC3'}
\begin{cases}
r_0u_r(r_0,\th)=1+f_0(\th),\\
\frac{1}{2}|\u(r_0,\th)|^2+p(r_0,\th)=b_0(\th),\\
\int_{r_0}^{r_1} u_\th(r,0)\,dr=j_0,\\
p(r_1,\th)=-\frac{1}{2r_1^2}+p_1(\th),
\end{cases}
\end{equation}
admits a unique solution  $(\u,p)\in (C^{1,\al}(\Om))^3$ if and only if the compatibility condition
\begin{equation}\label{BC3_comp}
p_1(0)=\ga_1^{BC3}(0)+\frac{1}{2r_1^2}
\end{equation}
is satisfied.
\end{thmp}
\begin{proof}
The Dirichlet boundary condition for the pressure $p$ in \eqref{BC3'}, when differentiated with respect to $\th$,  directly transforms into the last condition in \eqref{BC3}. According to Theorem \ref{thm3}, there exists a unique solution $(\u,p)\in (C^{1,\al}(\Om))^3$. Consequently, the value of pressure $p$ at the outer circle is determined by
\begin{equation}\label{p1}
\begin{split}
p(r_1,\th)&=p(r_1,0)+\int_0^{\th}\p_\th p_1(s)\,ds\\
&=-\frac{1}{2 r_1^2}+p_1(\th)+(\ga_1^{BC3}p)(0)-p_1(0)+\frac{1}{2r_1^2}.
\end{split}
\end{equation}

This implies that the pressure $p$ obtained by Theorem \ref{thm3} satisfies the boundary condition in \eqref{BC3'} if and only if $p_1(\th)$ fulfills the compatibility condition \eqref{BC3_comp}. This completes the proof.
\end{proof}

\section{The vorticity transport method}\label{sec3}
In this section, we investigate solutions to the Euler equation \eqref{InE_Pol} subject to boundary conditions \eqref{BC4} and \eqref{BC5}, focusing on perturbations around the reference flow $(\bar\u,\bar p)=(\frac{1}{r}\e_r,-\frac{1}{2 r^2})$. Through this section, we decompose the velocity field $\u$ as
\begin{equation*}
\u=\bar \u+\v.
\end{equation*}

Transforming the equation \eqref{eq_trans_om} to polar coordinates yields
\begin{equation}\label{eq_om_pol}
(u_r\p_r+\frac{u_\th}{r}\p_\th)\om=0.
\end{equation}

By applying the method of characteristics to the transport equation \eqref{eq_om_pol}, we may represent its solution as follows. Let $z_2(z_1;\th_0)$ denote the characteristic curve passing through the point $(z_1,z_2)$ with initial condition $z_2(r_0) = \th_0$. This curve is determined by the ordinary differential equation
\begin{equation*}
\begin{cases}
\frac{\p z_2}{\p z_1}(z_1;\th_0)=\frac{v_\th}{1+z_1v_r}(z_1,z_2(z_1;\th_0)),\\
z_2(r_0)=\th_0.
\end{cases}
\end{equation*}

Conversely, we may regard $\th_0$ as a function of $z=(z_1,z_2)$, which we denote by $\th_0=\th_0(z_1,z_2)$. Then the solution to \eqref{eq_om_pol} with initial data 
\begin{equation}\label{om0}
\om(r_0,\th)=\om_0(\th)
\end{equation}
is expressed as $\om(z)=\om_0(\th_0(z))$. Furthermore, we obtain the following results. A complete proof can be found in Proposition 3.8 of \cite{AV2022}.

\begin{lemma}\label{lem_trans}
Assume that $\om_0\in C^{1,\al}(\mb{T}_{2 \pi})$ and $\v\in C^{2,\al}(\Om)$ with $\|\v\|_{C^{2,\al}}\ll 1$, there is a unique solution $\om\in C^{1,\al}(\Om)$ to \eqref{eq_om_pol}-\eqref{om0} satisfying

\begin{equation}\label{est_om_1_al}
\|\om\|_{C^{1,\al}}\leq C(\al,r_0,r_1)\|\om_0\|_{C^{1,\al}}
\end{equation}
for some positive constant $ C(\al,r_0,r_1)>0$. Moreover, let $\om^i,\,i=1,2$, be two solutions of \eqref{eq_om_pol} with initial data $\om^i_0,\,i=1,2$, where the velocity field $\v$ in \eqref{eq_om_pol} is replaced by $\v^1$ and $\v^2$ respectively. Then, there exists a positive constant $C$ depending only on $\al,r_0,r_1$ such that
\begin{equation}\label{est_diff_om}
\|\om^1-\om^2\|_{C^\al}\leq C\big(
\|\om_0^1-\om_0^2\|_{C^\al}+\|\om^1_0\|_{C^{1,\al}}\|\v^1-\v^2\|_{C^\al}
\big).
\end{equation}
\end{lemma}

Having determined the vorticity field $\om$, we now turn to the analysis of the div-curl problem. Consider the following div-curl system in polar coordinates  
\begin{equation}\label{div_curl_pol}
\begin{cases}
\div_c\w:= \frac{1}{r}\p_r(r w_r)+\frac{1}{r}\p_{\th}w_\th=0,\\
\curl_c\w:=\frac{1}{r} \big( \p_r(r w_\th)-\p_\th w_r\big)=\om,
\end{cases}
\end{equation}
subject to the boundary conditions:
\begin{equation}\label{bc_dvs}
\begin{cases}
r_iw_r(r_i,\th)=f_0(\th),\,i=0,1\,,\\
\int_{r_0}^{r_1}w_\th(r,0)\,dr=j_0.
\end{cases} 
\end{equation}

Under the boundary condition \eqref{bc_dvs}, the well-posedness of equation \eqref{div_curl_pol} is established in the following lemma.

\begin{lemma}\label{lem_div_curl}
Let $\om\in C^{1,\al}(\Om),j_0\in\R$ and $f_i(\th)\in C^{1,\al}({\mb{T}_{2\pi}}),\,i=1,2$, satisfying
\begin{equation}\label{cond_flux_equi}
J_0:= \int_{0}^{2\pi}f_0(\th)\,d\th=\int_{0}^{2\pi}f_1(\th)\,d\th.
\end{equation}
Then, the div-curl problem \eqref{div_curl_pol}-\eqref{bc_dvs} admits a unique solution $\w\in C^{2,\al}(\Om)$ with the estimate
\begin{equation}\label{est_div_curl}
\|\w\|_{C^{k,\al}}\leq C(\al,r_0,r_1)\big(
\|\om\|_{C^{k-1,\al}}+\|f_0\|_{C^{k,\al}}+\|f_1\|_{C^{k,\al}}+|j_0|
\big),\,k=1,2,
\end{equation}
where $C=C(\al,r_0,r_1)$ is a positive constant.
\end{lemma}
\begin{proof}
Thanks to the divergence free condition $\div_c\w=0$, we study the following  elliptic equation:
\begin{equation}\label{eq_aux_phi}
\begin{cases}
\Delta_c\phi:=\frac{1}{r}\big(\p_r(r\p_r\phi)+\p_\th(\frac{1}{r}\p_\th\phi)\big)=-\om,\\
\phi(r_0,\th)=\phi_0(\th):=\int_0^{\th}(f_0(s)-f_0^{ave})ds,\\
\phi(r_1,\th)=\phi_1(\th):=-j_0+\int_0^{\th}(f_1(s)-f_1^{ave})ds,\\
\phi(r,\th+2\pi)=\phi(r,\th).
\end{cases}
\end{equation}
Since $\phi_i,i=0,1$, defined in \eqref{eq_aux_phi} belong to $C^{1,\al}(\mb{T}_{2\pi})$, the problem \eqref{eq_aux_phi} has a unique solution $\phi\in C^{3,\al}(\Om)$ with estimate
\begin{equation*}
\|\phi\|_{C^{k+1,\al}}\leq C\big(
\|\om\|_{C^{k-1,\al}}+\|f_0\|_{C^{k,\al}}+\|f_1\|_{C^{k,\al}}+|j_0|
\big),\,k=1,2.
\end{equation*}

Letting
$$rw_r=\p_{\th}\phi+\frac{J_0}{2\pi},\,w_\th=-\p_r\phi,$$
then it's easy to verify that $\v$ solves the div-curl problem \eqref{div_curl_pol}-\eqref{bc_dvs}, and satisfies the estimation \eqref{est_div_curl}.
\end{proof}

\subsection{Boundary condition \texorpdfstring{\eqref{BC4}}{}} \label{sec3.1}
In this subsection, we consider the steady Euler equation \eqref{InE_Pol} equipped with the following boundary condition:
\begin{equation}\tag{BC4}\label{BC4}
\begin{cases}
r_iu_r(r_i,\th)=1+f_i(\th),\,i=0,1,\\
\int_{r_0}^{r_1}u_\th(r,0)\,dr=j_0,\\
p(r_0,\th)=-\frac{1}{2 r_0^2}+p_0(\th).
\end{cases}
\end{equation}
The result is formally stated in the following theorem.
\begin{theorem}\label{thm4}
Given $f_0(\th),f_1(\th),p_0(\th)\in C^{2,\al}(\mb{T}_{2\pi})$ and the constant $j_0\in\R$. There exists a small constant $1>\vep_0>0$ such that for all $\vep\in(0,\vep_0]$, if
\begin{equation*}
\|f_0\|_{C^{1,\al}}+\|f_1\|_{C^{1,\al}}+\|p_0\|_{C^{1,\al}}+|j_0|\leq \vep,
\end{equation*}
then the steady incompressible Euler equation \eqref{InE_Pol} with the boundary condition \eqref{BC4} has a unique solution  $(\u,p)\in (C^{2,\al}(\Om))^3$. 
\end{theorem}

\begin{proof}
The proof is divided into four steps.

{\bf Step 1.} The design of the iteration scheme. Consider the Banach space
\begin{equation}\label{bc4_def_V}
V_\de:=\{\v\in C^{2,\al}_\star(\Om):\|\v\|_{C^{2,\al}}\leq \de\},
\end{equation}
here $\de>0$ is a small constant to be determined later. Fixing $\hat\v\in V_\de$,
we define $\om\in C^{1,\al}(\Om)$ as the unique solution of the following transport equation
\begin{equation}\label{bc4_eq_om}
\begin{cases}
\big( (\frac{1}{r}+\hat v_r)\p_r+\frac{\hat v_\th}{r}\p_\th)\om=0,\\
\om(r_0,\th)=\om_0(\th)
\end{cases}
\end{equation}
where 
\begin{equation}\label{bc4_om0}
\om_0(\th)=-\frac{1}{r_0^2}\p_\th f_0(\th)-\frac{1}{1+f_0}\p_\th p_0(\th)-\frac{1}{2(1+f_0)}\p_\th |\hat v_\th(r_0,\th)|^2.
\end{equation}
It follows from \eqref{est_om_1_al} that
\begin{equation}\label{bc4_est_om}
\begin{split}
\|\om\|_{C^{1,\al}}&\leq C
\bigg\|\frac{1}{r_0^2}\p_\th f_0(\th)+\frac{1}{1+f_0}\p_\th p_0(\th)+\frac{1}{2(1+f_0)}\p_\th |\hat v_\th(r_0,\th)|^2
\bigg\|_{C^{1,\al}}\\
&\leq C\bigg(
\|f_0\|_{C^{2,\al}}+(1+\|f_0\|_{C^{1,\al}})^2(\|p_0\|_{C^{p,\al}}+
\|\hat\v\|^2_{C^{2,\al}})
\bigg)\\
&\leq C(\vep+\de^2).
\end{split}
\end{equation}

Next, we define $\v$ solving the following div-curl problem:
\begin{equation}\label{bc4_eq_DV}
\begin{cases}
\frac{1}{r}\p_r(r v_\th)-\frac{1}{r}\p_\th v_r=\om,\\
\frac{1}{r}\p_r(r v_r)+\frac{1}{r}\p_{\th}v_\th=0,\\
r_iv_r(r_i)=f_i,\,i=0,1,\\
\int_{r_0}^{r_1}v_\th(r,0)\,dr=j_0.
\end{cases}
\end{equation}
Utilizing \eqref{est_div_curl} and \eqref{bc4_est_om}, we have
\begin{equation}\label{bc4_est_v}
\begin{split}
\|\v\|_{C^{2,\al}}&\leq C\big(
\|\om\|_{C^{1,\al}}+\|f_0\|_{C^{2,\al}}+\|f_1\|_{C^{2,\al}}+|j_0|
\big)\\
&\leq C(\vep+\de^2).
\end{split}
\end{equation}

Taking $\de=2C\vep$ and $\vep<\min\{1,\frac{1}{4 C^2}\}$, then the mapping $\ml{T}:V_\de\to V_\de$ given by
\begin{equation}\label{bc4_def_T}
\ml{T}\hat \v=\v
\end{equation}
is well defined due to the estimate \eqref{bc4_est_v}. 

{\bf Step 2.} We claim that the mapping $\ml{T}$ defined in \eqref{bc4_def_T} has a unique fixed point in $V_\de$. Since $V_\de$ is a closed subspace of $C^{1,\al}$, it is sufficient to show that $\ml{T}$ is contractible in the Banach space $C^{1,\al}(\Om)$.

$$\v  $$

For any given two vectors $\hat \v^i\in V_\de,\,i=1,2$, let $\v^i=\ml{T}\hat \v^i,\, i=1,2$. The estimate \eqref{est_diff_om} suggests that
\begin{equation}\label{bc4_est_diff_om}
\begin{split}
\|\om^1-\om^2\|_{C^\al}&\leq C\big(
\|\om_0^1-\om_0^2\|_{C^\al}+\|\om^1_0\|_{C^{1,\al}}\|\hat\v^1-\hat\v^2\|_{C^\al}
\big).\\
&\leq C\big(
\|(\hat v_\th^1-\hat v_\th^2)\p_\th \hat v_\th^1+(\p_\th\hat{v}_\th^1-\p_\th\hat{v}_\th^2)\hat v_\th^2\|_{C^\al}+
\vep \|\hat\v^1-\hat\v^2\|_{C^\al}
\big)\\
&\leq C\vep \|\hat\v^1-\hat\v^2\|_{C^{1,\al}}.
\end{split}
\end{equation}
The above inequality, together with the linearity of the div-curl system \eqref{bc4_eq_DV} and the estimate \eqref{est_div_curl} with $k=1$, leads to
\begin{equation}
\begin{split}
\|\ml{T}\hat\v^1-\ml{T}\hat\v^2\|_{C^{1,\al}}&=\|\v^1-\v^2\|_{C^{1,\al}}\leq C\|\om^1-\om^2\|_{C^\al}\\
&\leq C\vep \|\hat\v^1-\hat\v^2\|_{C^{1,\al}}.
\end{split}
\end{equation}
Therefore, there exists a small constant $\vep_0>0$ such that for any $0<\vep<\vep_0$, the mapping $\ml{T}$ is a contraction. Banach's fixed point theorem implies that $\ml{T}$ has a unique fixed point $\v\in V_\de$, which examines the claim.

{\bf Step 3.} The construction of pressure $p$. Assume that $\v\in V_\de$ is a fixed point of $\ml{T}$ and set $\u=\bar\u+\v$. The first equation in \eqref{bc4_eq_DV} yields to
\begin{equation*}
\curl_c\u=\curl_c\v=\om,
\end{equation*}
here $\om$ is the unique solution of the transport equation
\begin{equation*}
\u\cdot\nabla_c\,\om=0,\,\om_0=\curl_c\v\big|_{r=r_0}.
\end{equation*}

Introduce $\mf{G}=G_r\e_r+G_\th\e_\th$, where $G_r,G_\th$ are given by
\begin{equation}\label{def_G}
\begin{split}
G_r&=(\u\cdot\nabla_c)u_r-\frac{u_\th^2}{r},\\
G_\th&=(\u\cdot\nabla_c)u_\th+\frac{u_r u_\th}{r}.
\end{split}
\end{equation}

Then, a direct calculus by using the incompressibility condition $\div_c\u=0$ results in 
\begin{equation}\label{vort_vanish}
\curl_c\mf{G}=\frac{1}{r}\p_r(rG_\th)-\frac{1}{r}\p_\th G_r=\u\cdot \nabla_c\,\om=0.
\end{equation}

Therefore, we can define a scalar function $g(z)$ by
\begin{equation}\label{def_g}
g(z)=-\int_{\ga_z}\mf{G}\cdot d\mf{l},
\end{equation}
where $d\mf{l}=dr\e_r+rd\th \e_\th$ and $\ga_z$ is an arbitrary curve connecting $(r_0,0)$ with $(r,\th)$. In particular, one has $g(r_0,0)=0$.

By construction, we have 
\begin{equation}\label{eq_G_g}
\mf{G}+\nabla_c\,g=0.
\end{equation}
Moreover, we shall show that $g(z)$ remains invariant under translations by a period in the direction of $\e_\th$. It's sufficient to demonstrate that 
\begin{equation*}
\int_{0}^{2\pi} G_\th(r_0,\th)r_0 d\th=0.
\end{equation*}

By the definition, we have
\begin{equation}\label{om0_11}
\curl_c\u\big|_{r=r_0}=\frac{1}{r}(\p_r(ru_\th)-\p_\th u_r)\big|_{r=r_0}=
\big[\frac{1}{r}u_\th+\p_r u_\th\big]_{r=r_0}-\frac{1}{r_0^2}\p_\th f_0,  
\end{equation}
Notice that the initial condition $\om_0$  is given by
\begin{equation}\label{om0_12}
\curl_c\u\big|_{r=r_0}=\om_0=-\frac{1}{r_0^2}\p_\th f_0(\th)-\frac{1}{1+f_0}\p_\th p_0(\th)-\frac{1}{2(1+f_0)}\p_\th |u_\th(r_0,\th)|^2.
\end{equation}

Comparing \eqref{om0_11} with \eqref{om0_12}, we get 
\begin{equation}\label{deri_p0}
r_0G_\th(r_0,\th)=-\p_\th p_0(\th),
\end{equation}
and
\begin{equation*}
\int_{0}^{2\pi} r_0 G_\th(r_0,\th) d\th=-
\int_{0}^{2\pi} \p_\th p_0(\th) d\th=0.
\end{equation*}

Hence, the pressure $p$ can be defined as
\begin{equation}\label{bc4_def_p}
p(r,\th)=g(r,\th)+g_0,\,g_0=-\frac{1}{2r_0^2}+p_0(0).
\end{equation}

{\bf Step 4.} We check that $(\u,p)\in (C^{2,\al}(\Om))^3$ is a solution to the steady incompressible Euler equation \eqref{InE_Pol}, if and only if $\v\in V_\de$ is a fixed point of $\ml{T}$ and the pressure $p$ is constructed by \eqref{bc4_def_p}.

On the one hand, let $\v\in V_\de$ be a fixed point and set $\u=\bar\u+\v$. The definition of $p$ in \eqref{bc4_def_p} and \eqref{eq_G_g} yield to
\begin{equation}\label{eq_G_p}
{\mf{G}+\nabla_c p}=0,
\end{equation}
which are exactly the momentum equations to the Euler equation \eqref{InE_Pol}. We just examine the boundary condition of $p$ since the boundary conditions for $\u$ are immediately obtained. Using \eqref{bc4_def_p} and \eqref{deri_p0}, we obtain
\begin{equation*}
\begin{split}
p(r_0,\th)&=g_0+\int_0^\th\p_\th g(r_0,z_2)\,dz_2=g_0-
\int_0^\th G_r(r_0,z_2)\,r_0dz_2\\
&=-\frac{1}{2r_0^2}+p_0(0)+\int_0^\th \p_\th p_0(z_2)\,r_0dz_2=-\frac{1}{2r_0^2}+p_0(\th).
\end{split}
\end{equation*}

Conversely, assume that $(\u,p)\in (C^{2,\al}(\Om))^3$ is a solution to \eqref{InE_Pol} subject to the boundary condition \eqref{BC4}. Then, taking $\curl_c$ on the momentum equations in \eqref{InE_Pol} to derive
\begin{equation*}
\u\cdot\nabla_c(\curl_c\u)=(\bar\u+\v)\cdot\nabla_c(\curl_c\v)=0.
\end{equation*}
Observe that
\begin{equation*}
\curl_c\v=-\frac{1}{r u_r}(v_r\p_\th v_r+v_\th \p_\th v_\th+\p_\th p).
\end{equation*}
Substituting the boundary data in \eqref{BC4} into the above identity, we have
\begin{equation*}
\curl_c\v\big|_{r=r_0}=-\frac{1}{r_0^2}\p_\th f_0(\th)-\frac{1}{1+f_0}\p_\th p_0(\th)-\frac{1}{2(1+f_0)}\p_\th |\hat v_\th(r_0,\th)|^2.
\end{equation*}
It follows that 
\begin{equation*}
\curl_c\v=\om,\,\div_c\v=0.
\end{equation*}
By the uniqueness of the solution to the div-curl problem \eqref{bc4_eq_DV}, we conclude that $\v=\ml{T}\v$ is a fixed point. Hence, the proof is completed.
\end{proof}

\subsection{Boundary condition \texorpdfstring{\eqref{BC5}}{}} \label{sec3.2}
In this subsection, we investigate a boundary condition formulated by replacing the radial velocity at the outer circle with the tangential derivative of pressure. Alternatively, when an additional compatibility condition is satisfied, it becomes feasible to prescribe the pressure itself on the outer boundary instead of its derivative. Specifically, the first class of boundary conditions is given by
\begin{equation}\tag{BC5}\label{BC5}
\begin{cases}    
r_0u_r(r_0,\th)=1+f_0(\th),\\
\int_{r_0}^{r_1}u_\th(r,0)\,dr=j_0,\\
p(r_0,\th)=-\frac{1}{2 r_0^2}+p_0(\th),\\
\p_\th p(r_1,\th)=\p_\th p_1(\th).
\end{cases}
\end{equation}
The well-posedness is stated in the following theorem.
\begin{theorem}\label{thm5}
Given $f_0(\th),p_0(\th),p_1(\th)\in C^{2,\al}(\mb{T}_{2\pi})$ and the constant $j_0\in\R$. There exists a small constant $1>\vep_0>0$ such that for all $\vep\in(0,\vep_0]$, if
\begin{equation*}
\|f_0\|_{C^{1,\al}}+\|f_1\|_{C^{1,\al}}+\|p_0\|_{C^{1,\al}}+|j_0|\leq \vep,
\end{equation*}
then the steady incompressible Euler equation \eqref{InE_Pol} with the boundary condition \eqref{BC5} has a unique solution  $(\u,p)\in (C^{2,\al}(\Om))^3$. 
\end{theorem}
\begin{proof}
The proof follows the same general framework as Theorem \ref{thm4}, but requires significant modifications at each step. We now detail these necessary adaptations.

{\bf Step 1.} Since the value of radial velocity $v_r(r_1,\th)$ on the outer circle, which is essential for solving the div-curl problem, remains unspecified a priori, we must reconstruct it iteratively and accordingly define the iteration space as
\begin{equation}\label{def_mlV}
\ml{V}_\de:=\{(\v,f_1)\in C^{2,\al}_\star(\Om)\times C^{2,\al}(\mb{T}_{2\pi}):\|\v\|_{C^{2,\al}}+\|f_1\|_{C^{2,\al}}\leq \de\},
\end{equation}
here $\de>0$ is a small constant to be determined later. Fixing $(\hat\v,f_1)\in \ml{V}_\de$, we define $\om\in C^{1,\al}(\Om)$ as the unique solution of the following transport equation
\begin{equation}\label{bc5_eq_om}
\begin{cases}
\big( (\frac{1}{r}+\hat v_r)\p_r+\frac{\hat v_\th}{r}\p_\th)\om=0,\\
\om(r_0,\th)=\om_0(\th)
\end{cases}
\end{equation}
where 
\begin{equation}\label{bc5_om0}
\om_0(\th)=-\frac{1}{r_0^2}\p_\th f_0(\th)-\frac{1}{1+f_0}\p_\th p_0(\th)-\frac{1}{2(1+f_0)}\p_\th |\hat v_\th(r_0,\th)|^2.
\end{equation}
It follows from \eqref{est_om_1_al} that
\begin{equation}\label{bc5_est_om}
\begin{split}
\|\om\|_{C^{1,\al}}\leq C(\vep+\de^2).
\end{split}
\end{equation}

Next, we introduce a $2\pi$-periodic function $f_1(\th)$ in terms of $\hat\v$, $\om$ solving \eqref{bc5_eq_om} and boundary data. More precisely, we define
\begin{equation}\label{bc5_def_f1}
f_1(\th)=f_1(0)+\int_0^\th \ml{R}(z_2;\hat\v,\hat f_1)-\ml{R}^{ave}\,dz_2,
\end{equation}
where
\begin{equation}\label{def_R5}
\ml{R}(\th;\v(r,\th),f(\th))=-r_1^2\om(r_1,\th)-\frac{r_1^2}{1+f}\p_\th p_1(\th)-\frac{r_1^2}{2(1+f)}\p_\th |v_\th(r_1,\th)|^2.
\end{equation}

The mass flux of the velocity field $\v$ going through the outer circle is given by
\begin{equation*}
\int_0^{2\pi}f_1(\th)\,d\th
=2\pi f_1(0)+\int_{0}^{2\pi}\int_0^\th \ml{R}(z_2;\hat\v,\hat f_1)
-\ml{R}^{ave}\,dz_2\,d\th.
\end{equation*}
Since $\div_c\v=0$, the constant $f_1(0)$ is determined by
\begin{equation}\label{def_f1_0}
f_1(0)=\frac{1}{2\pi}\big( J_0+\int_{0}^{2\pi}
\big(\ml{R}(z_2;\hat\v,\hat f_1)-\ml{R}^{ave}\big)
(2\pi-z_2)\,dz_2 \big),
\end{equation}
where $J_0=\int_0^{2\pi}f_0(s)ds$. 

Using \eqref{bc5_est_om} and \eqref{bc5_def_f1}, we get 
\begin{equation}\label{bc5_est_f1}
\begin{split}
\|f_1\|_{C^{2,\al}}&\leq C\bigg(|f_1(0)|+|\ml{R}^{ave}|
+\|\ml{R}(\th;\hat v,\hat f_1)\|_{C^{1,\al}}\bigg)\\
&\leq \bigg(  |J_0|+\|\om\|_{C^{1,\al}}+(1+\|\hat f_1\|_{C^{1,\al}})^2
\big(\|p_1\|_{C^{2,\al}} +\|\hat\v\|_{C^{2,\al}}^2\big)
\bigg)\\
&\leq C(\vep+\vep\de+\de^2).
\end{split}
\end{equation}

Now, we can define $\v$ as the unique solution of the following div-curl problem
\begin{equation}\label{bc5_eq_DV}
\begin{cases}
\frac{1}{r}\p_r(r v_\th)-\frac{1}{r}\p_\th v_r=\om,\\
\frac{1}{r}\p_r(r v_r)+\frac{1}{r}\p_{\th}v_\th=0,\\
r_iv_r(r_i)=f_i,\,i=0,1,\\
\int_{r_0}^{r_1}v_\th(r,0)\,dr=j_0.
\end{cases}
\end{equation}

Utilizing \eqref{est_div_curl} again together with estimate \eqref{bc5_est_f1}, we derive
\begin{equation}\label{bc5_est_v_f}
\begin{split}
\|\v\|_{C^{2,\al}}+\|f_1\|_{C^{2,\al}}&\leq C\big(
\|\om\|_{C^{1,\al}}+\|f_0\|_{C^{2,\al}}+\|f_1\|_{C^{2,\al}}+|j_0|
\big)\\
&\leq C(\vep+\vep\de+\de^2).
\end{split}
\end{equation}

Choosing $\de=3C\vep$ and $\vep<\min\{1,\frac{1}{3C},\frac{1}{9C^2}\}$, the mapping $\ml{T}:\ml{V}_\de\to \ml{V}_\de$ given by
\begin{equation}\label{bc5_def_T}
\ml{T}\hat V=V,
\end{equation}
where $\hat V=(\hat \v,\hat f_1)$ and $V=(\v,f_1)$, is well defined by virtue of the estimate \eqref{bc5_est_v_f}. 

{\bf Step 2.} We examine the contraction of the mapping $\ml{T}$ in the Banach space $C^{1,\al}(\Om)$.

Given any two vectors $\hat V^i=(\hat\v^i,\hat f_1^i)\in \ml{V}_\de,\,i=1,2$, denote their images under the mapping $\ml{T}$ as $\hat V^i=\ml{T}V^i=(\v^i,f_1^)$ for $i=1,2$.

The definition of $\ml{T}$ and the estimate \eqref{est_div_curl} imply that
\begin{equation}\label{bc5_est_diff_V}
\begin{split}
\|V^1-V^2\|_{C^{1,\al}}&=\|\v^1-\v^2\|_{C^{1,\al}}+\|f_1^1-f_1^2\|_{C^{1,\al}}\\
&\leq C\big( \|\om^1-\om^2\|_{C^{\al}}+\|f_1^1-f_1^2\|_{C^{1,\al}} \big)    
\end{split}
\end{equation}

Following an analogous argument to the derivation of estimate \eqref{bc5_est_diff_om}, one can easily obtain
\begin{equation}\label{bc5_est_diff_om}
\begin{split}
\|\om^1-\om^2\|_{C^\al}\leq C\vep \|\hat\v^1-\hat\v^2\|_{C^{1,\al}}.
\end{split}
\end{equation}

In order to handle the last term in \eqref{bc5_est_diff_V}, we denote
\begin{equation*}
F^i(\th)=\ml{R}(\th;\hat\v^i,\hat f_1^i)-\ml{R}^{ave}(\th;\hat\v^i,\hat f_1^i),\,i=1,2.
\end{equation*}
It follows from \eqref{bc5_def_f1} and \eqref{def_R5} that
\begin{equation*}
f_1^1-f_1^2=\frac{1}{2\pi}\int_{0}^{2\pi}(F^1-F^2)(s)(2\pi-s)\,ds
+\int_{0}^{\th}(F^1-F^2)(s)\,ds,
\end{equation*}
and
\begin{equation*}
\begin{split}
\|f_1^1-f_1^2\|_{C^{1,\al}}&\leq C\|F^1-F^2\|_{C^{\al}}\leq 
C\|\ml{R}(\th;\hat\v^2,\hat f_1^2)-\ml{R}(\th;\hat\v^2,\hat f_1^2)\|_{C^{\al}}\\
&\leq C\bigg(
\|\om^1-\om^2\|_{C^{\al}}+
\bigg\|\frac{\p_\th p_1}{1+\hat f_1^1}-\frac{\p_\th p_1}{1+\hat f_1^2}\bigg\|_{C^{\al}}+
\bigg\|\frac{v_\th^1\p_\th v_\th^1}{1+\hat f_1^1}-\frac{v_\th^2\p_\th v_\th^2}{1+\hat f_1^2}\bigg\|_{C^{\al}}
\bigg)\\
&\leq C\big(
\|\om^1-\om^2\|_{C^{\al}}+\vep\|\hat f_1^1-\hat f_1^2\|_{C^\al}
+\vep\|\hat\v^1-\hat\v^2\|_{C^{1,\al}}
\big).
\end{split}
\end{equation*}

The above inequality together with \eqref{bc5_est_diff_om} and \eqref{bc5_est_diff_V} yields
\begin{equation}
\begin{split}
\|V^1-V^2\|_{C^{1,\al}}\leq C\vep
\|\hat\v^1-\hat\v^2\|_{C^\al}+C\vep\|\hat f_1^1-\hat f_1^2\|_{C^\al}
\leq C\vep \|\hat V^1-\hat V^2\|_{C^{1,\al}}.
\end{split}
\end{equation}
Consequently, there exists a small enough $\vep_0>0$  such that for all $0<\vep<\vep_0$, the mapping $\ml{T}$ becomes a contraction on $C^{1,\al}(\Om)$. By Banach's fixed point theorem, we conclude that T admits a unique fixed point $V=(\v,f_1)\in \ml{V}_\de$.

{\bf Step 3.} The pressure $p$ can be constructed following exactly the same procedure as {\bf Step 3} in the proof of Theorem \ref{thm4}, and therefore we omit the details here.

{\bf Step 4.} We only show that if $(\v,f_1)\in \ml{V}_\de$ is a fixed point of the mapping $\ml{T}$ defined in \eqref{bc5_def_T}, then $(\u,p)$ is a solution to \eqref{InE_Pol} with boundary condition \eqref{BC5}. The core task amounts to verifying the boundary condition of pressure on the outer circle, as all other requirements have been satisfied.

Since $(\v,f_1)$ is a fixed point of the mapping $\ml{T}$, we obtain
\begin{equation}\label{f1_p1}
\begin{split}
\p_\th f_1(\th)&=\ml{R}(\th;\v,f_1)-\ml{R}^{ave}(\th;\v,f_1)\\
&=-r_1^2\om(r_1,\th)-\frac{r_1^2}{1+ f_1}\p_\th p_1(\th)-\frac{r_1^2}{2(1+ f_1)}\p_\th |v_\th(r_1,\th)|^2-\ml{R}^{ave}\\
&=-\big[r^2\p_ru_\th+r u_\th+r\frac{u_\th}{u_r}\p_\th u_\th\big]_{r=r_1}
+\p_\th f_1(\th)
-\frac{r_1}{u_r}\p_\th p_1(\th)-\ml{R}^{ave}
\end{split}
\end{equation}
by using \eqref{bc5_def_f1}.

On the other hand, the second equation in \eqref{eq_G_p} gives
\begin{equation}\label{p_r1}
\begin{split}
-\p_{\th}p(r_1,\th)=r_1G_\th(r_1,\th) =\big[r_1u_r\p_ru_\th +u_\th\p_\th u_\th+u_ru_\th\big]_{r=r_1}
\end{split}
\end{equation}

Combining \eqref{f1_p1} and \eqref{p_r1}, we yield 
\begin{equation}\label{p1_R}
\p_\th p(r_1,\th)-\p_\th p_1(\th)=\frac{u_r(r_1,\th)}{r_1}\ml{R}^{ave}.
\end{equation}

Integrating the identity from $0$ to $2\pi$ gives
\begin{equation}\label{integ=0}
0= r_1^2\int_0^{2\pi}\p_\th(p(r_1,\th)-p_1(\th))d\th=\ml{R}^{ave}\int_0^{2\pi}(1+r_1v_r)d\th.
\end{equation}

Owing to the smallness of $\|\v\|_{C^{2,\al}}$, we conclude that $\ml{R}^{ave}$ must equal zero. Thus, \eqref{p1_R} is changed into
\begin{equation*}
\p_{\th}p(r_1,\th)=\p_{\th}p_0(\th),
\end{equation*}
which precisely corresponds to the required boundary condition of $p$ in \eqref{BC5}.
\end{proof}

If we replace the final boundary condition of $p$ in \eqref{BC5} with a Dirichlet boundary condition, resulting in the modified boundary data,
\begin{equation}\tag{BC5$'$}\label{BC5'}
\begin{cases}
r_0u_r(r_0,\th)=1+f_0(\th),\\
\int_{r_0}^{r_1}u_\th(r,0)\,dr=j_0,\\
p(r_0,\th)=-\frac{1}{2 r_0^2}+p_0(\th),\\
p(r_1,\th)=-\frac{1}{2 r_1^2}+p_1(\th),
\end{cases}
\end{equation}
the well-posedness remains valid provided that $p_1(\th)$ satisfies an appropriate compatibility condition. Similar to Theorem \ref{thm3'}, we introduce the trace operator $\ga^{BC5}_1$ as
\begin{equation*}
(\ga^{BC5}_1p)(\th)=p(r_1,\th),
\end{equation*}
where $p$ is the pressure derived from Theorem \ref{thm5}. We present without proof the following theorem regarding the solvability of the Euler equations subject to boundary condition \eqref{BC5'}.

\begin{thmp}
Let $f_0(\th),p_0(\th),p_1(\th)\in C^{2,\al}(\mb{T}_{2\pi})$ and the constant $j_0\in\R$.
There exists a small constant $1>\vep_0>0$ such that for all $\vep\in(0,\vep_0]$, if
\begin{equation*}
\|f_0\|_{C^{1,\al}}+\|f_1\|_{C^{1,\al}}+\|p_0\|_{C^{1,\al}}+|j_0|\leq \vep,
\end{equation*}
then the steady incompressible Euler equation \eqref{InE_Pol} with the boundary condition \eqref{BC5'} has a unique solution  $(\u,p)\in (C^{2,\al}(\Om))^3$ if and only if $p_1(\th)$ satisfies the following compatibility condition
\begin{equation*}
p_1(0)=(\ga_1^{BC5}p)(0)+\frac{1}{2 r_1^2}.
\end{equation*}
\end{thmp}

\subsection{Alternative approach for boundary condition \texorpdfstring{\eqref{BC1*},\,\eqref{BC2*} and \eqref{BC3}}{}}\label{sec3.3}

In this subsection, we prove Theorem \ref{thm3} again by utilizing the vorticity transport approach. First, it is convenient to rewrite the boundary condition \eqref{BC1} as follows:
\begin{equation}\tag{BC1$^*$}\label{BC1*}
\begin{cases}
r_0u_r(r_0,\th)=1+f_0(\th),\\
r_1u_r(r_1,\th)=1+f_1(\th),\\
\frac{1}{2}|\u(r_0,\th)|^2+p(r_0,\th)=b_0(\th),\\
\int_{r_0}^{r_1} u_\th(r,0)\,dr=j_0.
\end{cases}
\end{equation}
We have the following theorem.
\begin{theorem}\label{thm1*}
Let $f_0(\th),f_1(\th),b_0(\th)\in C^{2,\al}(\mb{T}_{2\pi})$ for all $\th\in\mb{T}_{2\pi}$. Then the steady incompressible Euler equation \eqref{InE_Pol} admits a unique solution $(\u,p)\in (C^{2,\al}(\Om))^3$ satisfying the boundary condition \eqref{BC1*} provided
\begin{equation*}
\|f_0\|_{C^{1,\al}}+\|f_1\|_{C^{1,\al}}+\|b_0\|_{C^{1,\al}}+|j_0|\leq \vep,
\end{equation*}
for all $\vep>0$ small enough.

\end{theorem}
\begin{proof}
The iteration scheme developed in the proof of Theorem \ref{thm4} still works if the initial condition of the transport type problem \eqref{bc4_eq_om} is given by
\begin{equation}\label{bc1*_om_0}
\om_0(\th)=-\frac{\p_{\th}b_0(\th)}{1+f_0(\th)}\in C^{1,\al}(\mb{T}_{2\pi}).
\end{equation}

In this setting, the unique existence of a fixed point $\v\in C^{2,\al}(\Om)$ can be verified straightforwardly. Therefore, we focus our discussion on the construction of the pressure $p$, and verification of the Bernoulli function's boundary condition at the inner circle.

From definitions \eqref{def_G} and \eqref{def_g}, together with the relation
\begin{equation*}
\mf{G}+\nabla_c g=0,
\end{equation*}
we derive through the second component that
\begin{equation}\label{bc1*_ber_deri}
\p_{\th}(\frac{1}{2}|\u|^2+g)=-ru_r\om.
\end{equation}

Integrating \eqref{bc1*_ber_deri} over the interval $(0,2\pi)$, it follows by using \eqref{bc1*_om_0} that
\begin{equation*}
(\frac{1}{2}|\u|^2+g)(r_0,\th)-\frac{1}{2}|\u|^2(r_0,0)=b_0(\th)-b_0(0).
\end{equation*}

We thus examined the boundary condition
\begin{equation}\label{bc1*_ber_r0}
(\frac{1}{2}|\u|^2+p)(r_0,\th)=b_0(\th)
\end{equation}
through the pressure definition:
\begin{equation}\label{bc1*_def_p}
p(r,\th)=g(r,\th)-\frac{1}{2}|\u(r_0,0)|^2+b_0(0).
\end{equation}

The proof is completed.
\end{proof}

Similarly, to show Theorem \ref{thm2} by the vorticity transport method, we represent the boundary condition \eqref{BC2} as
\begin{equation}\tag{BC2$^*$}\label{BC2*}
\begin{cases}
r_0u_r(r_0,\th)=1+f_0(\th),\\
\frac{1}{2}|\u(r_0,\th)|^2+p(r_0,\th)=b_0(\th),\\  
\frac{1}{2}|\u(r_0,\th)|^2+p(r_0,\th)=(b_0\circ T)(\th),
\end{cases}
\end{equation}
where $T:\mb{T}_{2\pi}\to \mb{T}_{2\pi}$ is a given orientation-preserving diffeomorphism with $C^{3,\al}$ regularity.

\begin{theorem}
Let $f_0(\th),b_0(\th)\in C^{1,\al}(\mb{T}_{2\pi})$. There exists a small constant $1>\vep_0>0$ such that for all $\vep\in(0,\vep_0]$, if
\begin{equation*}
\|T(\th)-\th\|_{C^{3,\al}}+\|f_0\|_{C^{2,\al}}+\|b_0\|_{C^{2,\al}} \leq \vep,
\end{equation*}
then the steady incompressible Euler equation \eqref{InE_Pol} with the boundary condition \eqref{BC2*} admits a unique solution $(\u,p)\in (C^{1,\al}(\Om))^3$.
\end{theorem}

\begin{proof}
Let $V_\de$ be the iteration space defined in \eqref{bc4_def_V}. For given $\hat\v\in V_\de$, let $\om$ denote the unique solution to the transport type problem \eqref{bc4_eq_om}, where the initial condition is replaced by 
\begin{equation}\label{bc2*_om_0}
\om_0(\th)=-\frac{\p_{\th}b_0(\th)}{1+f_0(\th)}\in C^{1,\al}(\mb{T}_{2\pi}).
\end{equation}

It follows from \eqref{est_om_1_al} that
\begin{equation}\label{bc2*_est_om}
\begin{split}
\|\om\|_{C^{1,\al}}\leq C\|\om_0\|_{C^{1,\al}}\leq C(1+\|f_0\|_{C^{1,\al}})^2\|b_0\|_{C^{2,\al}}.
\end{split}
\end{equation}

In order to deal with the div-curl system, we define a $2\pi$-periodic function $f_1(\th)$ 
\begin{equation}\label{bc2*_def_f1}
f_1(\th)=-1+T'(\th)+f_0(T(\th))T'(\th)\in C^{2,\al}(\Om).
\end{equation}

Then, we claim that the following div-curl problem
\begin{equation}\label{bc2*_div_curl}
\begin{cases}
\frac{1}{r}\p_r(r v_\th)-\frac{1}{r}\p_\th v_r=\om,\\
\frac{1}{r}\p_r(r v_r)+\frac{1}{r}\p_{\th}v_\th=0,\\
r_iv_r(r_i)=f_i(\th),\,i=0,1.
\end{cases}
\end{equation}
has a unique solution $\v\in {C^{2,\al}}(\Om)$. 

Indeed, by the standard theory of elliptic equations, the following boundary value problem
\begin{equation*}
\begin{cases}
-\Delta_c\vphi=\om,\\
\vphi(r_0,\th)=\int_0^{\th} f_0(s)\,ds-\bar J_0\th,\\
\vphi(r_1,\th)=-\th+T(\th)+\int_0^{T(\th)}f_0(s)\,ds-\bar J_0\th,\\
\vphi(r,\th+2\pi)=\vphi(r,\th),
\end{cases}
\end{equation*}
admits a unique solution $\vphi\in C^{3,\al}(\Om)$ satisfying
\begin{equation*}
\|\vphi\|_{C^{k+1,\al}}\leq C\big(\|T(\th)-\th\|_{C^{k+1,\al}}+\|T\|_{C^{k+1,\al}}\|f_0\|_{C^{k,\al}}+
\|f_0\|_{C^{k,\al}}+\|\om\|_{C^{k-1,\al}}\big),
\end{equation*}
for $k=1,2$ and some positive constant $C>0$.

The solution $\v$ to the div-curl problem \eqref{bc2*_div_curl} is constructed by
\begin{equation*}
rv_r=\p_{\th}\vphi+\bar J_0 \th,\,v_\th=-\p_r\vphi.
\end{equation*}

Moreover, $\v$ satisfies the following estimate:
\begin{equation}\label{est_bc2*_v}
\begin{split}
\|\v\|_{C^{k,\al}}\leq C \big(\|T(\th)-\th\|_{C^{k+1,\al}}+
\|f_0\|_{C^{k,\al}}+\|\om\|_{C^{k-1,\al}}\big),\,k=1,2.
\end{split}
\end{equation}

By using \eqref{bc2*_est_om} and \eqref{est_bc2*_v}, it's valid to define the operator $\ml{T}:V_\de\to V_\de$ as $\ml{T}\hat\v=\v$. The contraction property of the mapping $\ml{T}$ follows directly from \eqref{est_diff_om}. For brevity, we omit the detailed verification here.

With the pressure function $p$ defined in \eqref{bc1*_def_p}, we immediately obtain that the Bernoulli function satisfies its boundary condition at the inner circle, while the boundary data on the outer circle is ensured by the formula \eqref{check_B}. This completes the proof of the theorem.
\end{proof}

We now prove Theorem \ref{thm3} using the vorticity transport method.

{\bf Proof of Theorem \ref{thm3}.} The iteration scheme follows essentially the same construction as in the proof of Theorem \ref{thm5}, with the only modification being the initial condition for the transport equation \eqref{bc5_eq_om}, which is now specified as
\begin{equation}\label{bc3_om0}
\om_0(\th)=-\frac{\p_\th b_0(\th)}{1+f_0(\th)}.
\end{equation}

Through a line-by-line adaptation of {\bf Step 1-2} in the proof of Theorem \ref{thm5}, we derive that there is a unique fixed point $(\v,f_1)$. Let the pressure $p$ be determined by \eqref{bc1*_def_p}. Then, the boundary condition for the Bernoulli function in \eqref{BC3} is automatically satisfied by construction. The boundary condition for the pressure derivative in \eqref{BC3},
\begin{equation*}
\p_\th p(r_1,\th)=\p_\th p_1(\th),
\end{equation*}
is verified by analogous arguments to those developed in {\bf Step 4} of the proof of Theorem \ref{thm5}, thereby completing our proof.

\vspace{10pt}

{\bf Data Availability Statement.} No data, models or code were generated or used during the study.

{\bf Conflict of interest.} On behalf of all authors, the corresponding author states that there is no conflict of interests.


\end{document}